\documentclass[11pt,amssymb,verbatim]{psp}

\usepackage{amssymb,graphics,graphicx,amsfonts} 
\usepackage{amscd} 
\usepackage[utf8]{inputenc}
\usepackage[T1]{fontenc}
\usepackage{amsmath}
\usepackage{amsfonts,color}
\usepackage{amssymb}
\usepackage{graphicx,hyperref} 
\usepackage{enumerate} 

\usepackage{graphicx,tikz} 
\newtheorem{theoreme}{Theorem}[section] %
\newtheorem{proposition}[theoreme]{Proposition} %
\newtheorem{lemme}[theoreme]{Lemma} %
\newtheorem{definition}{Definition}[section] %
\newtheorem{remark}[theoreme]{Remark} %
\newtheorem{exemple}[theoreme]{Exemple} %

\newcommand{\rad}{\rho_1}

\begin{document}

\title[An anisotropic inhomogeneous ubiquity theorem]
  {An anisotropic inhomogeneous ubiquity theorem}

\author[E. Daviaud]{\'Edouard DAVIAUD\\
 Universit\'e Paris-Est, LAMA (UMR 8050)\addressbreak UPEMLV, UPEC, CNRS, F-94010, Cr\'eteil, France}

\volume{1}
\pubyear{2021}
\setcounter{page}{1}
\receivedline{Received \textup{12} October \textup{2021;}
              revised \textup{...} ...l \textup{...}}
\maketitle

\section{Introduction}
Let $(B_n)_{n\in\mathbb{N}}$ be a sequence of balls in $\mathbb{R}^{d}$, $d\ge 1$, endowed with any norm $\|\cdot\|$. Starting from some simple geometric property of the set of points falling in infinitely many of the $B_n$'s, i.e. $\limsup_{n\rightarrow+\infty}B_n$,  finding estimates for the Hausdorff dimension of the limsup sets of shrunk versions of $(B_n)_{n\in\mathbb{N}}$ into smaller sets  is a natural and old question, which has been studied in depth. The first result goes back to Jarnik and Besicovitch, who proved that for every $\tau\geq 1$, the dimension of the set $\bigcap_{k\in\mathbb{N}^{*}}\bigcup_{\substack{q\in \mathbb{N}^{*}\\ q\geq k}}\bigcup_{p\in\mathbb{Z}}B({p}/{q},1/q^{2\tau})$ has Hausdorff dimension~$\frac{1}{\tau}$ (although the result was not stated  in terms of sequence of balls such that the limsup has full Lebesgue measure, the proof uses explicitly this geometric fact).


It was first established by Jaffard \cite{Ja} that if $\limsup_{n\rightarrow +\infty}B_n$ has full Lebesgue measure, then for every $\tau\geq 1$,  the Hausdorff dimension  of $\limsup_{n\rightarrow +\infty}B_n ^{\tau}$ (where for a closed ball $B=B(x,r)$ of center $x$ and radius $r\geq 0$, the ball $B^{\tau}$ is defined by $B^{\tau}=B(x,r ^{\tau})$) is bounded by below as follows:
  $$\dim_{H}(\limsup_{n\rightarrow +\infty}B_n ^{\tau})\geq\frac{d}{\tau}.$$
Thanks to this result, Jaffard was able to compute the multifractal spectrum of certain lacunary wavelet series \cite{Ja}. This so-called ubiquity result was generalized by Dodson $\&$ al. in \cite{Dodson}, where the notion of ubiquitous system is introduced, and further refined by Beresnevitch and Velani in \cite{BV}. Given a metric space $X$ and an Ahlfors regular  Radon measure $\mu$ (i.e.   there exists $\alpha\geq 0$ such that for every ball $B$ of radius $r$ small enough, one has $C^{-1}r^{\alpha}\leq\mu(B)\leq Cr^{\alpha}$  for some uniform constant $C>0$), Bersenevich and Velani prove that as soon as $\limsup_{n\rightarrow+\infty}B_n$ has full $\mu$-measure, then one has $\mathcal{H}^{\frac{\alpha}{\tau}}(\limsup_{n\rightarrow+\infty} B_n ^{\tau})=+\infty$, where $\mathcal{H}^{\frac{\alpha}{\tau}}$ denotes the Hausdorff measure of dimension $\frac{\alpha}{\tau}$. Different approaches, using various distribution properties of the centres of the balls $B_n$, were also developed intensively (see the monographs \cite{Bugeaud} and \cite{Durand}).  

%
The inhomogeneous case, i.e when the information about $\limsup_{n\rightarrow +\infty}B_n$ is not given by the Lebesgue measure, or an Ahlfors regular one as in \cite{Ja} and \cite{BV},  but rather by a multifractal measure possessing scale invariance like properties, has been studied by Barral and Seuret in \cite{BSCMP,BS}. For instance, they proved that for a quasi-Bernoulli probability measure $\mu$ (see Definition~\ref{QB}), if $\mu(\limsup_{n\rightarrow +\infty}B_n)=1$, then the same type of result stands. Namely, if ${\dim}(\mu)$ denotes the dimension of the measure $\mu$ (see Definition~\ref{exa} and Proposition~\ref{QBexa} below), then one has  
\begin{equation}\label{ubiBS}
\dim_{H}(\limsup_{n\rightarrow +\infty}B_n ^{\tau})\geq \frac{{\dim} (\mu)}{\tau}.
\end{equation}
This type of result has many applications to the multifractal analysis of functions, measures and capacities (see, e.g., \cite{BSCMP,BSJIMJ,BS2020}). 

Recently, in an other direction,  Wang and Wu, working with the $\|  \cdot   \|_{\infty}$ norm, dealt with the anisotropic case, when the balls (which are Euclidean cubes) are shrunk into   thin rectangles and when the reference measure is the Lebesgue measure (or an Alhfors regular one) in \cite{RE}. More precisely, for any sequence of balls $(B_n=\prod _{i=1}^{d}[x_n^i,x_n^i+r_n])_{n\in\mathbb{N}}$ in $(\mathbb R^d,\| \ \|_{\infty})$, given  $1\leq \tau_1\leq \tau_2 \leq ...\leq \tau_d$, these authors consider the collection consisting of the $B_n$'s shrunk into rectangles defined by 
\begin{equation}
\label{defRn}
R_n=\prod _{i=1}^{d}[x_n^i,x_n^i+r_n ^{\tau_i}], \ \ \ \mbox{for every } n\in\mathbb{N}.
\end{equation} 
They proved that if $\limsup_{n\rightarrow +\infty}B_n$ has full Lebesgue measure, then
\begin{equation}
\label{rectleb}\dim_{H}(\limsup_{n\rightarrow+\infty}R_n)\geq\min_{1\leq i\leq d}\left(\frac{d+\sum_{1\leq j\leq i}\tau_i-\tau_j}{\tau_i}\right).\end{equation}
Later, based on the remark that the technique used in \cite{RE} carries a certain form of genericity, Rams and Koivusalo were able to deduce a general principle of computation for balls shrunk into sets of arbitrary (open) shapes in \cite{KR}.

\medskip

The present paper aims at shedding some light on how anistropic settings can be handled within the inhomogeneous case. As a consequence of our main result, following the previous notations, we obtain that if  $\limsup_{n\rightarrow +\infty}B_n$ has full measure for a quasi-Bernoulli measure $\mu$ fully supported on $[0,1]^d$, then 
\begin{equation*}
\dim_{H}(\limsup_{n\rightarrow+\infty}R_n)\geq\min_{1\leq i\leq d}\left (\frac{\dim (\mu)+\sum_{1\leq j\leq i}\tau_i-\tau_j}{\tau_i}\right ).
\end{equation*}
\section{Preliminaries and statement of the main result}

\subsection{Some notations.}

The space $\mathbb{R}^{d}$ is endowed with the infinity norm $\|\,\|_\infty$. 

For $x\in\mathbb{R}^{d}$ and $r\ge 0$,  $B(x,r)$ stands for the closed ball of center $x$ and radius~$r$, and for $t\geq 0$ and $\tau\in\mathbb{R}$, setting $B=B(x,r)$,  $t B$ and  $B^{\tau}$ denote the balls $B(x,t r)$ and $B(x ,r^{\tau})$ respectively.

If $E\subset \mathbb{R}^d$, $\mathring{E}$ and $\partial E$ denote its interior and its boundary, $|E|$ its diameter, and if $E$ is a Borel set, $\mathcal{B}(E)$ denotes the trace of the Borel $\sigma$-algebra $\mathcal{B}(\mathbb{R}^d)$ on $E$. Also, $\dim_H (E)$ and $\dim_{P}(E)$ respectively denote the Hausdorff dimension and the packing dimension of $E$  (see, e.g., \cite{F} for the definitions). 

$\mathcal{L}^d$ stands for the Lebesgue measure on $(\mathbb{R}^d,\mathcal{B}(\mathbb{R}^d))$, and $\mathcal{P}([0,1]^d)$ stands for the set of Borel probability measures on $([0,1]^d,\mathcal B([0,1]^d))$. For $\mu \in\mathcal{P}([0,1]^d)$, one denotes by $\supp(\mu)$ the topological support of $\mu$.


 $\mathcal{M}_{d}(\mathbb{R})$ and $\mathcal{O}_{d}(\mathbb{R})$ are   the space of $d\times d$ real matrices and  the group of orthogonal matrices of $\mathcal{M}_{d}(\mathbb{R})$.
 
If $r_1,\ldots,r_d$ are $d$ real numbers,   \mbox{diag}$(r_1,...,r_d)$ stands for the diagonal matrix  $A \in\mathcal{M}_d (\mathbb{R})$ such that $A_{i,j}=r_i\delta_{i,j}$ for all $1\leq i, j\leq d$, where $\delta_{i,j}=1$ if $i=j$ and $\delta_{i,j}=0$ otherwise.

Given $\mu\in\mathcal{P}([0,1]^d)$ and $T:[0,1]^d\mapsto [0,1]^d$  a measurable function, one defines 
$T\mu=\mu\circ T^{-1}.$

For $p\in\mathbb{N}$,  $\mathcal{D}_p$ stands for  the set of closed dyadic subcubes of $[0,1]^d$ of generation $p$, i.e $$\mathcal{D}_{p}=\left\{\prod_{i=1}^{d}[k_i 2^{-p}, (k_i +1)2^{-p}] \  : \ \forall1\leq i\leq d , \ 0\leq k_{i}\leq 2^{p}-1\right\}.$$
For $D\in\mathcal D_p$, we also denote $p$ by $p(D)$. Observe  that $D\in\mathcal{D}_{p(D)}.$

\subsection{Some definitions and recalls}
\begin{definition}
Let $\mu\in\mathcal{P}([0,1]^d)$. For $x\in \supp(\mu)$, the local lower and upper dimensions of $\mu$ at $x$ are $$\underline\dim_{{\rm loc}}(\mu,x)=\liminf_{r\rightarrow 0^{+}}\frac{\log(\mu(B(x,r)))}{\log(r)}$$ and $$\overline\dim_{{\rm loc}}(\mu,x)=\limsup_{r\rightarrow 0^{+}}\frac{\log (\mu(B(x,r)))}{\log(r)}.$$ \medskip
One also sets  $\underline{\dim}_H (\mu)={\rm essinf}_{\mu}(\underline\dim_{{\rm loc}}(\mu,x))$ and $\overline{\dim}_P (\mu)={\rm esssup}_{\mu}(\overline\dim_{{\rm loc}}(\mu,x)).$

 \medskip

\end{definition}

It is known that (see \cite{F} for instance)
$$\underline{\dim}_H (\mu)=\inf\{\dim_{H}(E): \, E\in\mathcal{B}([0,1]^d), \ \mu(E)>0\}
$$ 
and 
$$
\overline{\dim}_P (\mu)=\inf\{\dim_{P}(E):\, E\in\mathcal{B}([0,1]^d), \ \mu(E)=1\}.
$$

\begin{definition}
\label{exa}
A measure $\mu\in\mathcal{P}([0,1]^d)$ is said to be exact dimensional if there exists $\alpha\in\mathbb R_+$ such that for $\mu$-almost all $x\in[0,1]^d$, one has  $\overline\dim_{{\rm loc}}(\mu,x)=\underline\dim_{{\rm loc}}(\mu,x)=\alpha$, i.e. $\underline{\dim}_H (\mu)=\overline{\dim}_P (\mu)=\alpha$. In this case $\alpha$ is simply denoted by $\dim (\mu)$. 
\end{definition}
%

We now define quasi-Bernoulli measures associated with the dyadic cubes (our main results easily extend to the case of $b$-adic cubes).
\begin{definition}
\label{QB}
Let $\mu\in\mathcal{P}([0,1]^d)$. For $D\in\mathcal{B}([0,1]^d)$ such that $\mu(D)>0$, define  $$\mu_D=\frac{\mu_{\vert_D}}{\mu(D)}.$$

When $D$ is a closed dyadic subcube of $[0,1]^d$,   $T_{D}:D\rightarrow [0,1]^d$ stands for the canonical affine mapping which sends $D$ onto $[0,1]^d$. In addition, when $\mu(D)>0$ one defines 
 $$
 \mu^{D}=T_{D}\mu_D\in \mathcal{P}([0,1]^d).$$

The measure $\mu$ is said to be  quasi-Bernoulli when there exists a constant $C_{\mu}\geq 1$ such that for every $p\in\mathbb{N}$ and every $D\in\mathcal{D}_{p}$ with $\mu(D)>0$, one has
\begin{equation}
\label{qua}
\frac{1}{C_{\mu}}\mu\leq\mu^{D}\leq C_{\mu}\mu.
\end{equation} 
\end{definition}
The measure $\mu_D$ is the renormalized restriction of $\mu$ to $D$ and $\mu^{D}$ is the rescaled version of $\mu_D$ on the unit cube.
\medskip
\begin{exemple}
Define $\Lambda=\left\{0,1 \right\}$, $\Sigma =\Lambda^{\mathbb{N}}$, $\sigma$ be the shift operator on  $\Sigma$, and endow $\Sigma$ with the standard ultra-metric distance. Let $\pi$ the canonical projection of $\Sigma$ onto $[0,1]$. For any  H\"older potential $\phi$ on $\Sigma$, denote by $\nu_{\phi}$ the unique  equilibrium state associated with $\phi$ on $\Sigma$ (see \cite{Bowen}). Then the measure $\mu_{\phi} =\nu_{\phi} \circ \pi ^{-1}$ is quasi-Bernoulli, and $\nu_\phi$ is also called a Gibbs measure associated with $\varphi$. This follows from the fact that there exists a number $P(\varphi)$, the topological pressure of $\varphi$, and $C\geq 1$, such that for all $x\in\Sigma$, for all $n\in\mathbb N$:
$$
C^{-1}\le \frac{\nu_\phi\left(\{y=(y_i)_{i=1}^\infty\in\Sigma: \, y_i=x_i \text{ for all $1\le i\le n$}\}\right)}{e^{-nP(\varphi)+\sum_{k=0}^{n-1}\varphi(\sigma^k x)}}\le C.
$$
Note that there exist quasi-Bernoulli measures obtained as projections of  measures of Gibbs type  associated to potentials $\phi$ with much weaker regularity properties (see \cite{Walters,BKM}).
\end{exemple}

\begin{remark}\label{nomassonboundary}
It is easily seen that a quasi-Bernoulli measure $\mu$, if not supported on an affine hyperplane, is such that  $\mu(\partial [0,1]^d)=0$. For otherwise its orthogonal projection onto at least one of the sets $\{0\}^i\times[0,1]\times\{0\}^{d-i-1}$, which is quasi-Bernoulli as well, would have an atom at $(0,\ldots,0)$ or $(\underbrace{0,\ldots,0}_i,1,\underbrace{0,\ldots,0}_{d-i-1})$. This should  imply that it is a Dirac mass, hence $\mu$ is supported on a hyperplane. This property will be used in the proof of our main result.
\end{remark}

Let us recall the following result.
\begin{proposition}[\cite{H}]\label{QBexa}
A quasi-Bernoulli probability measure is exact dimensional.
\end{proposition}
\subsection{Main statement}

Our main result is the following. Recall our notations \eqref{defRn} for $(R_n)_{n\in \mathbb{N}}$.

\begin{theoreme}
\label{lowerbound}
Let  $\mu \in\mathcal{P}([0,1]^d)$ be a quasi-Bernoulli probability measure fully supported on $[0,1]^d$. Let $(B_n := B(x_n ,r_n))_{n\in\mathbb{N}}$ be a sequence of balls in $[0,1]^d$ such that $\lim_{n\rightarrow+\infty}r_n =0$ and $\mu(\limsup_{n\rightarrow+\infty}B_n)=1.$

Let $1\leq \tau_1\leq ...\leq \tau_d$ be $d$ real numbers, $\boldsymbol{\tau}=(\tau_1,\ldots,\tau_d)$ and $(O_n)_{n\in\mathbb{N}}\in\mathcal{O}_{d}(\mathbb{R})^{\mathbb{N}}$ be a sequence of orthogonal matrices. For $n\in\mathbb N$, set 
\begin{equation}
\label{rn}
R_n=x_n+O_n \widetilde{R}_n, \text{ where }\widetilde{R}_n={\rm diag}(r_n^{\tau_1},...,r_n^{\tau_d})\cdot [0,1]^d
\end{equation}
and
\begin{equation}
\label{smutau}
s(\mu,\boldsymbol{\tau})=\min_{1\leq k\leq d}\left( \frac{\underline{\dim}_H (\mu)+\sum_{1\leq j\leq k}\tau_k-\tau_j}{\tau_k}\right).
\end{equation}
One has 
\begin{equation}
\label{min}
\dim_H(\limsup_{n\rightarrow +\infty}R_n)\geq s(\mu,\boldsymbol{\tau}).
\end{equation}
\end{theoreme}

\begin{remark}
(1) For convenience, in particular to follow the point of view adopted in \cite{RE}, the results are stated with $\mathbb{R}^d$ endowed with $\| \ \cdot \ \|_{\infty}$ and for balls shrunk into rectangles with one vertex equal to the center of the shrunk ball. However, we emphasize that, up to very slight modifications of the proof (essentially by adding constants at some places), they still hold for another norm and if   the balls are shrunk into rectangles or ellipsoids containing the center of the initial cube.\medskip 

(2) Given $\tau>1$, by taking $\tau_i=\tau$ for all $1\le i\le d$ and $O_n=I_d$ for all  $n\in\mathbb N$, , Theorem \ref{lowerbound} reduces to Barral-Seuret's theorem  \cite{BS} in the special case of quasi-Bernoulli measures, i.e~\eqref{ubiBS}.\medskip

(3) By taking $\mu=\mathcal{L}^d$ and $O_n=I_d$ for all  $n\in\mathbb N$, we recover the result established in~\cite{RE}, i.e.,  formula \eqref{rectleb}.
\end{remark}

\begin{remark}
The proof does not entirely use the exact dimensionality of $\mu$, the key property is the quasi-Bernoulli property \ref{qua}. However, the fact that $\underline{\dim}_H(\mu)=\overline{\dim}_H(\mu)$ can be used to prove $\dim_H(\limsup_{n\rightarrow +\infty}R_n)\leq s(\mu,\boldsymbol{\tau})$ under additional assumptions. The existence of   upper bounds for the Hausdorff dimension of limsup of sets (e.g. of rectangles)  included in balls $(B_n)_{n\geq N}$  will be  achieved in an independent paper,   in a   general setting.
\end{remark}

\section{Proof of theorem \ref{lowerbound}}
 Fix once and for all  the quasi-Bernoulli measure $\mu$, $1\leq \tau_1\leq ...\leq \tau_d$ and $\boldsymbol{\tau}=(\tau_1 ,...,\tau_d)$. Recall that  $\alpha=\dim_{H}(\mu)$ is the dimension of $\mu$.


The lower bound of Theorem \ref{lowerbound} is  obtained by constructing a Cantor set included in $\limsup_{n\rightarrow+\infty}R_n$, and  of dimension larger than or equal to $s(\mu,\boldsymbol{\tau})$. Before starting the construction,  two helpful results are recalled.
\begin{proposition}[Mass distribution principle, see \cite{F}]~ \ \label{MD}

 Let $A\in\mathcal{B}(\mathbb{R}^d)$ and $\mu\in\mathcal{M}(\mathbb{R})^d$. Suppose that there exists $C>0$ and $r>0$, $0\leq s\leq d$, such that for every ball of $\mathbb{R}^d$ $B=B(x,r^{\prime})$ with $r^{\prime}<r$, $\mu(B)\leq C(r^{\prime})^{s}$. Then $\mathcal{H}^{s}(A)\geq \frac{\mu(A)}{C}$.
 In particular, if $\mu(A)>0$ then $\dim_H (A)\geq s.$
\end{proposition} 
The second one is a classical technical lemma.
\begin{lemme}
\label{geo}
Let $A=B(x ,r)$ and $B=B(x^{\prime} ,r^{\prime})$ be two  closed balls, and $q\geq 3$  be such that $A\cap B\neq \emptyset$ and $A\setminus (qB)\neq \emptyset$. Then $r^{\prime}\leq r$ and $qB\subset 5A.$
\end{lemme}

\begin{proof}
Consider $z\in A\setminus qB$. One has
\[qr^{\prime}\leq \| z-x^{\prime}\|_{\infty}\leq \| z-x\|_{\infty}+\| x-x^{\prime}\|_{\infty}\leq r+r+r^{\prime}.\]
Hence $\frac{q-1}{2}r^{\prime}\leq r$, and in particular, one necessarily has $r^{\prime}\leq r$ and $qr^{\prime}\leq 2r+r^{\prime}\leq 3r.$

Furthermore, if $y\in qB$, then
\[\| y-x \|_{\infty}\leq \| x^{\prime}-y \|_{\infty}+\| x^{\prime}-x \|_{\infty}\leq qr^{\prime}+r^{\prime}+r\leq 5r.\]
This concludes the proof.
\end{proof}
 We construct thereafter a Cantor set $K$ as well as a sequence of strictly positive real numbers $(\varepsilon_{p})_{p\in\mathbb{N}}$ and a Borel probability measure $\eta$ such that:
\bigskip
\begin{itemize}
\item[•] $K\subset\limsup_{n\rightarrow +\infty} R_n$ and $\eta(K)=1$,\medskip
\item[•] The sequence $(\varepsilon_p)_{p\in\mathbb{N}}$ is decreasing with $\lim_{p\rightarrow +\infty}\varepsilon_p =0$ and there exists a constant $ C$ such that for any $p\in\mathbb{N}$, there exists $r_{p}>0$ verifying, for any ball $B\subset \mathbb{R}^d$ of radius $r$ less than $r_p$,  
\begin{equation}
\label{eqfinale}
\eta(B)\leq C.r^{s(\mu,\boldsymbol{\tau})-4\varepsilon_p}.
\end{equation}
\end{itemize}
Then, applying the mass distribution principle (Proposition \ref{MD}), since $\eta(K)=1$ one deduces that, for any $p>0$, $$\dim_H (\limsup_{n\rightarrow +\infty}R_n)\geq \dim_H (K)\geq s(\mu,\boldsymbol{\tau})-4\varepsilon_p,$$    
and letting $p\to+\infty$ concludes the proof.

\medskip

The construction of $(K,\eta)$ is decomposed into several steps. Without loss of generality we assume that $s(\mu,\boldsymbol{\tau})>0$. Fix a decreasing sequence $(\varepsilon_p)_{p\in\mathbb{N}}$ converging to 0 at $\infty$,   such that $\varepsilon_0 \leq   \max(1, s(\mu,\boldsymbol{\tau})/4)$.   

\subsection*{Step 1: Initialization} Let us start with a definition.
\begin{definition}
For $\nu\in\mathcal{P}([0,1]^d)$, $\beta\geq 0$, and $\varepsilon,\rho>0$, define  
$$
E_{\nu}^{\beta,\varepsilon,\rho}=\left\{x\in[0,1]^d:  \forall \, 0<r\leq \rho, \ B(x,r)\subset [0,1]^d\text{ and }  \nu(B(x,r))\leq r^{\beta-\varepsilon}\right\}.
$$
Then set 
$$ 
E_{\nu}^{\beta,\varepsilon}=\bigcup_{n\geq 1}E_{\nu}^{\beta,\varepsilon,\frac{1}{n}}.
$$
\end{definition}


With $\beta=\alpha=\underline{\dim}_H(\mu)$, since $\mu(\partial [0,1]^d)=0$ (due to Remark~\ref{nomassonboundary} and the assumption that $\mu$ is fully supported), for all $\varepsilon>0$, one has  $\mu(E^{\alpha,\varepsilon}_{\mu})=1$.  For all $p\in\mathbb{N}$, consider $\rho_{p}\in (0,1)$  small enough so that  
\begin{equation}
\label{equalab}
\mu(E_{\mu}^{\alpha,\varepsilon_p,\rho_{p}})\geq \frac{1}{2}.
\end{equation}
\medskip

Now, recall the following covering theorem due to Besicovitch(\cite{BES}):
\begin{theoreme}
\label{Besi}
There exists a positive integer $Q_{d}$, depending only on the dimension $d$, such that for every $E\subset [0,1]^{d}$,  for every set $\mathcal{F}=\left\{B(x, r(x) ): x\in E,  r(x) >0 \right\}$, there are $\mathcal{F}_1,...,\mathcal{F}_{Q_{d}}$ finite or countable collections of balls all contained in  $\mathcal{F}$ such that: 
\begin{itemize}
  \item each family $\mathcal{F}_i$ is composed of pairwise disjoint balls, i.e 
$\forall 1\leq i\leq Q_{d}$, $L\neq L'\in\mathcal{F}_i$, one has $L \cap L'=\emptyset,$ 

\item 
$E$ is covered by the families $\mathcal{F}_i$, i.e.
\begin{equation}\label{besi}
 E\subset  \bigcup_{1\leq i\leq Q_{d}}\bigcup_{L\in \mathcal{F}_i}L.
 \end{equation}
\end{itemize}
\end{theoreme}


 For $x \in E_{\mu}^{\alpha,\varepsilon_1,\rad}\cap \limsup_{n\rightarrow +\infty}B_n $,  consider $n_{x}\geq 1$ large enough   so that 
$x \in B_{n_x}$, $4r_{n_x}\leq \rad$, and 
\begin{equation} 
\label{rajou1}
 r_{n_x}^{-\varepsilon_1} \geq \max \left\{ 4Q_d4^{\alpha-\varepsilon_1} , \rho_2^{-d/\tau_d}\right\}.
 \end{equation} 
  Set 
 \begin{equation}
 \label{Lball}
 L_{x}=B(x,4r_{n_x}). 
\end{equation} 
 Doing so for every $x\in E_{\mu}^{\alpha,\varepsilon_1,\rad}\cap \limsup_{n\rightarrow +\infty}B_n $ provides us with a Besicovith covering $\mathcal F^1=\left\{L_x:  x\in E_{\mu}^{\alpha,\varepsilon_1,\rad}\cap \limsup_{n\rightarrow +\infty}B_n \right\}$ such that for every $x$, the ball $L_x$ is naturally associated with an integer $n_x \ge 1$ such that $x\in B_{n_x}$ and $| L_x |=8r_{n_x}.$ Also,   the shrunk rectangle $R_{n_x}$ verifies $R_{n_x}\subset B_{n_x}\subset L_x.$ This is illustrated by Figure \ref{fig1}.

Applying now Theorem \ref{Besi},  from the family $\mathcal F^1$ one can extract $Q_d$ finite or countable families of balls $\mathcal{F}^{1}_i$ , $1\leq i\leq Q_d$, such that:\medskip
\begin{itemize}
\item[•] $\forall \, 1\leq i\leq Q_d, \ \forall L\neq L^{\prime}\in\mathcal{F}_i ^{1}$, it holds that $L\cap L^{\prime}=\emptyset,$\medskip
\item[•] $E_{\mu}^{\alpha,\varepsilon_1,\rad}\cap \limsup_{n\rightarrow +\infty}B_n\subset \bigcup_{1\leq i\leq Q_d}\bigcup_{L\in \mathcal{F}^1_i}L.$\medskip
\end{itemize}
Since $\mu\left(E_{\mu}^{\alpha,\varepsilon_1,\rad}\cap \left(\limsup_{n\rightarrow +\infty}B_n \right)\right)\geq\frac{1}{2}$, there exists   $1\leq i_1\leq Q_d$ such that 
$$\mu\Big(\bigcup_{L\in\mathcal{F}_{i_1}}L\Big)\geq  \frac{\mu(E_{\mu}^{\alpha,\varepsilon_1,\rad}\cap \big(\limsup_{n\rightarrow +\infty}B_n \big))}{Q_d}\geq \frac{1}{2Q_d}.$$
Denote by $(L_{k}^{(1)})_{k\in\mathbb{N}}$ the sequence of balls such that $\mathcal{F}^1 _{i_1}=\left\{L_{k}^{(1)} \right\}_{k\in\mathbb{N}}$, $(x_{k}^{(1)})_{k\in\mathbb{N}}$ the sequence of points such that for all $k \in\mathbb{N}$, $L_{k}^{(1)}=L_{x_{k}^{(1)}}$, and set $r_{k}^{(1)}=r_{x_{k}^{(1)}}.$ 
There exists $N_1 \in \mathbb{N}$ so that
$$\mu\Big(\bigcup_{1\leq k\leq N_1}L_{k}^{(1)}\Big)\geq\frac{\mu\left(\bigcup_{L\in\mathcal{F} ^1 _{i_1}}L\right)}{2}.$$
Set $\mathcal{F}_1 =\left\{L_{k}^{(1)} \right\}_{1\leq k\leq N_1}$. One has
\begin{equation}
\label{cov 1}
\mu\Big(\bigcup_{L\in\mathcal{F}_1}L\Big)\geq \frac{1}{4Q_d}.
\end{equation}
Recall that with every ball $L_{k}^{(1)}$ are naturally associated the ball $B_{n_{k}^{(1)}}$ and  the rectangle~$R_{n_{k}^{(1)}}$, where $n_{k}^{(1)}=n_{x_{k}^{(1)}}$; set $R^{(1)}_{k}=R_{n_{k}^{(1)}}$.  Then define $K_1$, the first generation of the Cantor set by setting
$$
\mathcal{K}_1 =\left\{ R^{(1)}_{k}\right\}_{1\leq k\leq N_1}\text{ and }
K_1 =\bigcup_{R\in\mathcal{K}_1}R.
$$
Finally,  measure $\eta_1$ on the algebra generated by $\mathcal{K}_1$ is obtained by concentrating the $\mu$-measure  of the balls $L_x$ on the rectangle $R_{n_x}$. More precisely, for $1\leq k\leq N_1$ set
$$\eta_1(R_{k}^{(1)})=\frac{\mu\left(L_{k}^{(1)}\right)}{\sum_{1\leq k^{\prime}\leq N_1}\mu\big(L^{(1)}_{k^{\prime}}\big)}.$$
 Since for all $1\leq k\leq N_1$, the center $x_{k}^{(1)}$ of $L^{(1)}_k$ belongs to $E_{\mu}^{\alpha
 ,\varepsilon_1,\rad}$, recalling that
 $| L_{x_{k}^{(1)}} |/2=4r_{n_{k}^{(1)}}\leq \rad$, the disjointness of the $L_{j}^{(1)}$, as well as  the inequality \eqref{cov 1}, we get that for all $1\leq k\leq N_1$,
\begin{equation}
\label{minomesrec}
\eta_1(R_{k}^{(1)})\leq 4Q_d\left(4r_{n_{k}^{(1)}}\right)^{\alpha-\varepsilon_1}\leq \left(r_{n_{k}^{(1)}}\right)^{\alpha-2\varepsilon_1},
\end{equation}
where \eqref{rajou1} has been used.

\subsection*{Step 2: Constructing the second generation} 
This step consists of two sub-steps: First we associate a set of dyadic cubes with each rectangle previously obtained, and then we work inside each of these cubes.

\subsection*{Sub-step 2.1: A set of dyadic cubes inside each $R$ of $\mathcal K_1 $} 

\medskip
Consider a rectangle $R$. There exists an orthogonal matrix $O\in\mathcal{O}_{d}(\mathbb{R})$, a point $x\in [0,1]^d$ and  $0<\ell_d \leq \ell_2 \leq...\leq \ell_1$ such that
 $$R=x+O \widetilde R, \text{ with }\widetilde R=\prod_{i=1}^{d}[0,\ell_i].$$
 
Set $p=-\left\lfloor \log_2 \left(\frac{\ell_d}{8\sqrt{d}}\right)\right  \rfloor$. Intuitively, $2^{-p}\approx \frac{\ell_d}{8\sqrt{d}}$, so that there are some cubes included in $R$ with side-length $2^{-p}.$ We associate  with $R$ the set of dyadic cubes
$$\mathcal{C}(R)=\left\{D\in\mathcal{D}_{p}  :  D\subset  R , \ D=\prod_{i=1}^{d}[k_i 2^{-p},(k_{i}+1)2^{-p}], \ 8 | k_i , \ \forall\, 1\leq i\leq d\right\} .$$

%
%
%

Observe that  $\mathcal{C}(R)$ consists in dyadic cubes of generation $p$ inside $R$   that are quite far from each other. This will ensure that the rectangles used at a given generation of the construction of the Cantor set are well separated. Also, there exist a constant $C_d\ge 1$ depending only on the dimension $d$, such that the side length $2^{-p}$ of each $C\in \mathcal C(R)$ satisfies $C^{-1}_d \ell_d\le 2^{-p}\le C_d \ell_d$, as well as a constant $\kappa_d\geq 1$ such that
$$\kappa_d ^{-1}\prod_{i=1}^{d}\frac{\ell_i}{\ell_d}\leq\#\mathcal{C}(R)\leq \kappa_d \prod_{i=1}^{d}\frac{\ell_i}{\ell_d}.$$

Recalling \eqref{rn}, for every $n\in\mathbb{N}$, one gets
   \begin{equation}
   \label{minoC}
   \kappa_{d}^{-1}\cdot  r_{n}^{\sum_{i=1}^{d}\tau_i-\tau_d}\leq \#\mathcal{C}(R_n)\leq \kappa_{d}\cdot r_{n}^{\sum_{i=1}^{d}\tau_i-\tau_d}.
   \end{equation}
 
Now we  construct a measure $\eta_2$, which refines the measure $\eta_1$ by distributing the mass uniformly between the cubes of $\mathcal{C}(R)$ for $R\in \mathcal{K}_1$. For every $1\leq k\leq N_1$ and every $D\in\mathcal{C}(R_{k}^{(1)})$, set 
 $$\eta_2(D)=\frac{\eta_1(R_{k}^{(1)})}{\#\mathcal{C}(R_{k}^{(1)})}.$$
 By construction, $\eta_2 (R_{k}^{(1)})=\eta_1 (R_{k}^{(1)}).$
  Recalling  \eqref{minomesrec} and \eqref{minoC}, one gets
 \begin{align}
  \label{majocarre}
\eta_2(D) = \frac{\eta_1(R_{k}^{(1)})}{\#\mathcal{C}(R_{k}^{(1)})}\leq  \frac{\Big(r_{n_{k}^{(1)}}\Big)^{\alpha-2\varepsilon_1}}{\kappa_{d}^{-1} \cdot \Big(r_{n_{k}^{(1)}}\Big)^{\sum_{i=1}^{d}-\tau_d+\tau_i}}
= \kappa_{d}\cdot  \Big(r_{n_{k}^{(1)}} ^{\tau_d}\Big)^{\frac{\alpha-2\varepsilon_1+\sum_{i=1}^{d}\tau_d-\tau_i }{\tau_d} }.
\end{align}

\subsection*{Sub-step 2.2: Construction in each cube of $\mathcal C(R)$} We start with preliminary observations about the measure $\mu$. Recall Definition~\ref{QB}. Since $\mu$ is a quasi-Bernoulli measure, for every  $q\in\mathbb{N}$, every $D\in\mathcal{D}_q$ such that $\mu(D)>0$, for every $x\in[0,1]^d$ and $r>0$ such that $B(x,r)\subset D$, due to \eqref{qua} one has
\begin{align*}
\mu(B(x,r))=\mu\left(T_{D} ^{-1}(T_{D} (B(x,r)))\right) 
&=\mu(D) \mu^{D}\left(B\left(T_D (x),\frac{r}{ 2^{-q}}\right )\right)\\ 
&\leq C_\mu\, \mu(D)\, \mu\left(B\left(T_D (x),\frac{r}{ 2^{-q}}\right )\right).\nonumber
\end{align*}
Thus, for all $x\in [0,1]^d$ and $r>0$ such that $B(x,r)\subset [0,1]^d$ one has
\begin{equation}
\label{minoqb}
 \mu\left(B(T_{D}^{-1}(x),r 2^{-q})\right)\leq C_\mu\, \mu(D)\mu(B(x,r)).
 \end{equation}
Also, for every $p\in\mathbb{N}$,
 \eqref{qua} yields  
 \begin{equation}\label{muD}
 \mu^{D}(E_{\mu}^{\alpha,\varepsilon_p,\rho_{p}})\geq \frac{\mu(E_{\mu}^{\alpha,\varepsilon_p,\rho _{p}})}{C_{\mu}}\geq \frac{1}{2C_{\mu}}.
 \end{equation}
Moreover, 
  \begin{align*}
  &T_{D} ^{-1}(E_{\mu}^{\alpha,\varepsilon_p,\rho _{p}})\\ &=\left\{T_{D}^{-1}(x): \ \forall\, r\leq \rho_p , \,  B(x,r)\subset[0,1]^d,\, \mu(B(x,r))\leq r^{\alpha -\varepsilon_p}\right\} \\
  &=\left\{T_{D}^{-1}(x): \ \forall\, r\leq \rho _{p}, \,  B(x,r)\subset[0,1]^d,\, \mu\left(T_{D}\left(B\left(T_{D}^{-1}(x), \frac{r}{ 2^{q}}\right)\right)\right)\leq r^{\alpha -\varepsilon_p}\right\},
  \end{align*}
and using \eqref{minoqb}, one gets
  \begin{align*}
    &T_{D} ^{-1}(E_{\mu}^{\alpha,\varepsilon_q,\rho_q}) \\
    &\subset\left\{T_{D}^{-1}(x): \ \forall\, r\leq \rho _{p}, \,B(x,r)\subset[0,1]^d,\, \frac{\mu(B(T_{D}^{-1}(x),r 2^{-q}))}{\mu(D)}\leq C_{\mu}r^{\alpha -\varepsilon_p}\right\}\\
    &=\left\{y\in D: \ \forall\, r\leq \rho _{p} 2^{-q}, \, B(y,r)\subset D,\, \frac{\mu(B(y,r))}{\mu(D)}\leq C_{\mu}\Big(\frac{r}{2^{-q}}\Big)^{\alpha -\varepsilon_p}\right\}.
  \end{align*}
It follows that if we fix $p$  as above and set
  \begin{align}
  \label{E_D}
  E^{\varepsilon_p}_D= \limsup_{n\rightarrow +\infty}B_n \ \cap
  \left\{y\in D: \ \forall \,r\leq \rho _{p} 2^{-q}, \, B(y,r)\subset D,\, \frac{\mu(B(y,r))}{\mu(D)}\leq C_{\mu}\Big(\frac{r}{2^{-q}}\Big)^{\alpha -\varepsilon_p}\right\},
    \end{align}
then  by Definition \ref{QB} and the fact that $\mu(\limsup_{n\rightarrow+\infty}B_n)=1,$ we have
  \begin{equation}
 \label{rajou2}
  \mu(E^{\varepsilon_p}_D)=\mu(T_D ^{-1}(T_D (E^{\varepsilon_p}_D)))=\mu(D)\mu^D(T_D (E^{\varepsilon_p}_D))\ge \mu(D)\frac{\mu(E_{\mu}^{\alpha,\varepsilon_p,\rho _{p}})}{C_{\mu}}\geq\frac{\mu(D)}{2C_{\mu}},
  \end{equation}
  where we used \eqref{muD}.
  
  
 \medskip
 
We now continue the construction.  Consider $R\in\mathcal{K}_1$. Fix $D\in\mathcal{C}(R)$. Recall that $p(D)$ is the unique integer such that $D\in \mathcal{D}_{p(D)}$.  The set  $E_D ^{\varepsilon_2}$ is well defined since $\mu(D)>0$ (the measure~$\mu$ has been supposed to be fully supported on $[0,1]^d$). 
 For every $x\in E_D ^{\varepsilon_2}$, consider $n_x$ large enough so that:
 
 \medskip
 $\bullet$
 $x\in B_{n_x}$,
\medskip

$\bullet$ $n_x\geq 2$ and 
 \begin{equation} 
 \label{posi} 
 4r_{n_x}\leq  \rho_{2} 2^{-p(D)}, \ \text{ and } \  r_{n_x}^{-\varepsilon_2}\geq \max\left\{ 4C_{\mu}Q_d \cdot \eta_2(D)  (4\cdot 2^{p(D)})^{\alpha-\varepsilon_2},\rho_{3}^{-d/\tau_d}\right\}.
 \end{equation}
Set $L_x=B(x,4r_{n_x})$, as in step 1 (see \eqref{Lball}). By repeating the same argument as in step 1, one can extract from $\left\{L_x  :  x\in  E_D ^{\varepsilon_2}\right\}$
 a finite number $N_D$ of balls, $L_{1}^{(D)}=L_{x_{1}^{
 (D)}},...,L_{N_D }^{(D)}=L_{x_{N_{D}}^{(D)}}$ such that for all 
 $ 1\le k_1\neq k_2 \leq N_{D}$ one has $L_{k_1}^{(D)}\cap L_{k_2}^{(D)}=\emptyset$ and by \eqref{rajou2}
 \begin{equation}
 \label{recouL} 
\mu\Big(\bigcup_{1\leq k\leq N_{D}}L_{k}^{(D)}\Big)\geq \frac{\mu(E_D ^{\varepsilon_2})}{2} \geq 
 \frac{\mu(D)}{4Q_d C_{\mu}}.
\end{equation} 
 and with each ball $L_{k}^{(D)}$ are associated the ball $B_{n_{x_{k}^{(D)}}}$ and the rectangle $R_{n_{x_{k}^{(D)}}}$, that we denote by $B^{(D)}_{k}$ and $R_{k}^{(D)}$ respectively; we also set $r_k^{(D)}=r_{n_{x_k^{(D)}}}$. Then define the collection of rectangles of second generation by setting 
  $$\mathcal{K}(R)=\bigcup_{D\in\mathcal{C}(R)}\left\{R_{k}^{(D)} \right\}_{ 1\leq k\leq N_{D}}\text{ and }
  \mathcal{K}_2=\bigcup_{R\in \mathcal{K}_1}\mathcal{K}(R),
  $$
  and 
  $$
  K_2=\bigcup_{R\in\mathcal{K}_2}R.
  $$
  One extends further the measure $\eta_1$ to the algebra generated by the elements of the set  $\mathcal{K}_1 \bigcup\bigcup_{R\in\mathcal{K}_1}\mathcal{C}(R)\bigcup \mathcal{K}_2$  by distributing the mass according to $\mu$ at that scale. More precisely, for all $R\in\mathcal{K}_1$, $D\in\mathcal{C}(R)$ and $1\leq k\leq N_D$, one sets
  $$\eta_2(R_{k}^{(D)})=\eta_2(D) \frac{\mu(L_{k}^{(D)})}{\sum_{1\leq k^{\prime}\leq N_D}\mu(L_{k^{\prime}}^{(D)})}$$

 Note the following facts:\medskip 
  \begin{itemize} 
 \item[•] If $R\in \mathcal{K}_1 $, $D,D^{\prime}\in\mathcal{C}(R)$,  $1\leq k \leq N_D$ and $1\leq k' \leq N_{D^{\prime}}$ are such that \medskip $R_{k }^{(D)} \neq R_{k' }^{(D^{\prime})}\in \mathcal{K}(R)$, then $3B_{k}^{(D)}\cap 3B_{k' }^{(D^{\prime})}=\emptyset.$\medskip
  \item[•] If  $R\in\mathcal{K}_1$, $D\in\mathcal{C}(R)$ and $1\leq k\leq N_D$, using the second assertion of \eqref{posi} and the fact that the ball $L_{k}^{(D)}$ is centered on  $E_D ^{\varepsilon_2}$, then
  $$\frac{\mu(L_{k}^{(D)})}{\mu(D)}\leq C_{\mu} \left(\frac{4r_{k}^{(D)}}{ 2^{-p(D)}}\right)^{\alpha-\varepsilon_2} $$
  so that by \eqref{recouL} and the third assertion of \eqref{posi}, we get
  \begin{equation}
  \label{majorect1}
  \eta_2(R_{{k}}^{(D)})\leq\big ( \eta_2(D) 4Q_d C_{\mu} (4\cdot 2^{p(D)})^{\alpha -\varepsilon_2}\big ) \cdot (r_{k}^{(D)})^{\alpha -\varepsilon_2}\leq (r_{k}^{(D)})^{\alpha-2\varepsilon_2}.
\end{equation}  
    \end{itemize}

\subsection*{ Further steps: Induction scheme} We proceed as in step 2. Suppose that $ p\geq 2$, and  for all $1\leq q\leq p $, a set $K_q$ and a measure $\eta_q$, defined on the algebra generated by the elements of $\bigcup_{1\leq p\leq q}\mathcal{K}_p \bigcup_{R\in\mathcal{K}_p}\mathcal{C}(R)$, have been constructed in such a way that \eqref{minomesrec} holds and:\medskip
\begin{enumerate}
\item[\textbf{(i)}] For all $1\leq q\leq p$, $\mathcal{K}_q$ is a finite subset of $\left\{R_n \right\}_{n\geq q}.$\medskip
\item[\textbf{(ii)}] For all $ 2\leq q\leq p$,  for all $R\in\mathcal{K}_q$, there exists $R^{\prime}\in \mathcal{K}_{q-1}$ and $D\in\mathcal{C}(R^{\prime})$ such that $R\subset D$; one denotes by $\left\{R_{k}^{(D)}\right\}_{1\leq k\leq N_D}$ the family of rectangles of $\mathcal{K}_q$ included in~$D$.

\medskip
\item[\textbf{(iii)}] For all $1\leq q \leq p-1$ and $R\in\mathcal{K}_q $, if $r^{\tau_d}$ is the length of the smallest side of $R$, then
\begin{equation}
\label{sub}
(r^{\tau_d})^{-\varepsilon_{q}}\geq \rho_{q+1}^{-d}.
\end{equation}
\item[\textbf{(iv)}] For all $2\leq q\leq p$,  $R\in\mathcal{K}_{q-1} $, $D\in\mathcal{C}(R)$ and $1\leq k\leq N_D$, with the rectangle $R_{k}^{(D)}$ are naturally associated a point $x_{k}^{(D)}\in E_D^{\varepsilon_q}$, a ball  $L_{k}^{(D)}=B\left (x_{k}^{(D)}, 4r_k^{(D)}\right )$, as well as  some integer  $n_k\in\mathbb N$, such that $n_k\ge q$, $x_{k}^{(D)}\in B_k^{(D)}:=B_{n_k}=B(x_{n_{k}},r_{n_k})$, $R_{k}^{(D)}=R_{n_k}$,   $r_k^{(D)}=r_{n_{k}}$ and $4 r_k^{(D)}\leq 2^{-p(D)} \rho_{q}$. In particular, due to \eqref{E_D}, one has
\begin{equation}
\label{majoraf}
\frac{\mu(L_{k}^{(D)})}{\mu(D)}\leq C_\mu \left( \frac{4r_{n_{k}^{(D)}}}{2^{-p(D)}}\right)^{\alpha-\varepsilon_q}.
\end{equation}
\medskip
 \item[\textbf{(v)}] For all $ 2\leq q\leq p$,  $R\in \mathcal{K}_{q-1}$, $D,D^{\prime}\in\mathcal{C}(R)$, $1\leq k\leq N_D$ and  $1\leq k'\leq N_{D'}$ such that $R_{k }^{(D)}\neq R_{k'}^{(D^{\prime)}}$, one has $3B_{k}^{(D)}\cap 3B_{k'}^{(D^{\prime})}=\emptyset.$
\medskip
\item[\textbf{(vi)}] For all $1\leq q\le q' \leq p$ and $R\in\mathcal{K}_q$, $\eta_{q}(R)=\eta_{q'} (R).$\medskip 
\medskip
\item[\textbf{(vii)}] For all $2\leq q\leq p$,  $R\in \mathcal{K}_{q-1}$, $D\in\mathcal{C}(R)$ and $1\leq k\leq N_D$, one has
\begin{equation}
\label{defCf}
\eta_{q}(D)= \frac{\eta_{q-1}(R)}{\#\mathcal{C}(R)}\text{ and } (r_{k}^{(D)})^{-\varepsilon_q}\geq 4C_{\mu}Q_d\cdot \eta_q (D) (4\cdot 2^{p(D)})^{(\alpha-\varepsilon_q)}.
\end{equation}
\medskip
\item[\textbf{(viii)}] For all $2\leq q\leq p$,  $R\in \mathcal{K}_{q-1}$, $D\in\mathcal{C}(R)$ and $1\leq k\leq N_D$,  one has
\medskip
\begin{equation}
\label{CoverLf}
\sum_{1\leq k^{\prime}\leq N_D}\mu(L_{k^{\prime}}^{(D)})\geq \frac{\mu(D)}{4Q_d}
\end{equation}
and 
\begin{equation}
\label{defmesRf}
\eta_{q}(R^{(D)}_{k})=\eta_q (D)\cdot \frac{\mu(L_{k}^{(D)})}{\sum_{1\leq k^{\prime}\leq N_{D}}\mu(L_{k^{\prime}}^{(D)})}.
\end{equation}
\end{enumerate}

Notice that  by \eqref{majoraf},  \eqref{defCf} \eqref{CoverLf} and \eqref{defmesRf}, for all $ 2\leq q \leq p$, $R\in\mathcal{K}_{q-1}$, $D\in\mathcal{C}(R)$  and $1\leq k\leq N_D $, one has
  \begin{align}
  \label{majorect}
 \eta_{q}(R^{(D)}_{k})&=\eta_q (D)\cdot \frac{\mu(L_{k}^{(D)})}{\sum_{1\leq k^{\prime}\leq N_{D}}\mu(L_{k^{\prime}}^{(D)})}\leq \eta_q (D)\frac{\mu(L_{k}^{(D)})}{\mu(D)(4Q_d)^{-1}}\nonumber\\
 &\leq \eta_q(D)\Big(4r_{k}^{(D)}\Big)^{\alpha -\varepsilon_q}4Q_d C_{\mu}\Big(2^{p(D)}\Big)^{\alpha-\varepsilon_q} \leq \Big(r_{k}^{(D)}\Big)^{\alpha-2\varepsilon_q}.
\end{align}  
Thus, for all $ 2\leq q \leq p$, $R\in\mathcal{K}_{q-1}$ and $D\in\mathcal{C}(R)$, denoting by $r^{\tau_d}$ the length of the smallest side of $R$, by  \eqref{minomesrec}, \eqref{minoC}, \eqref{defCf},\eqref{majorect} and \textbf{(vi)}, one has
\begin{align}
\label{majomesC}
\eta_{q}(D)= \frac{\eta_{q-1}(R)}{\#\mathcal{C}(R)}\leq \frac{r^{\alpha-2\varepsilon_{q-1}}}{\kappa_{d}^{-1} \cdot r^{\sum_{i=1}^{d}-\tau_d+\tau_i}}&\leq \kappa_{d} r^{\alpha-2\varepsilon_{q-1}+\sum_{i=1}^{d}\tau_d-\tau_i }\nonumber \\
&\leq \kappa_{d} \Big(r^{\tau_d}\Big)^{\frac{\alpha-2\varepsilon_{q-1}+\sum_{i=1}^{d}\tau_d-\tau_i }{\tau_d} }.
\end{align}
Let us now explain the induction. Take $R\in \mathcal{K}_p$ and $D\in\mathcal{C}(R)$.
For every $x\in E^{\varepsilon_{p+1}}_{D}$, consider an integer $n_{x}$ large enough so that:\medskip
\begin{itemize}
\item[•] $x\in B_{n_x}$,

\item[•] $n_x \geq p+1$, $4r_{n_x}\leq \rho_{p+1} 2^{-p(D)}$, and 
$$
r_{n_x}^{-\varepsilon_{p+1}}\ge \max\left (
4^{\alpha -\varepsilon_{p+1}}\frac{\eta_p (R)}{\# \mathcal{C}(R)} 4Q_d C_{\mu} 2^{p(D)(\alpha-\varepsilon_{p+1})}, \rho_{p+2}^{-d/\tau_d}\right).
$$
\end{itemize}
Using Besicovitch covering Theorem \ref{Besi}, one can extract from the covering of  $E^{\varepsilon_{p+1}}_{D}$, 
 $\left\{L_{x}:=B(x,4r_{n_x}):  x\in E^{\varepsilon_{p+1}}_{D} \right\}$,
 a finite set of balls $\mathcal{F}(D):=\left\{L_{k}^{(D)}:=L_{x_{k}}^{(D)}  \right\}_{1\leq k\leq N_{D}}$ such that\medskip
 \begin{itemize}
 \item[•] $\forall k\neq k' \leq N_{D}$, $L_{k}^{(D)}\cap L_{k'}^{(D)}=\emptyset.$ In particular, $3B_{n_{x_{k}^{(D)}}}\cap 3B_{n_{x_{k'}^{(D)}}}=\emptyset,$\medskip
 \item[•]one has
  
\begin{equation}
\label{coverc}
 \mu\Big(\bigcup_{1\leq k\leq N_{D}}L_{k}^{(D)}\Big)\geq \frac{1}{2} \mu(E^{\varepsilon_{p+1}}_{D}) \geq 
 \frac{\mu(D)}{4Q_d C_{\mu}}.
 \end{equation}
 \medskip
  
  \end{itemize}
  Consider the  collection of rectangles naturally associated with the balls  {$L_{k}^{(D)}$}
  $$\mathcal{K}_{p+1}(R)=\bigcup_{D\in\mathcal{C}(R)}\left\{R_{k}^{(D)}:=R_{n_{x_{k}^{(D)}}}\right\}_{1\leq k\leq N_{D}}.$$
Then define 
  $$\mathcal{K}_{p+1}=\bigcup_{R\in \mathcal{K}_p} \mathcal{K}(R)\text{ and } K_{p+1}=\bigcup_{R\in\mathcal{K}_{p+1}}R.$$
  
 The probability  measure $\eta_p $ can be extended from the algebra generated by the elements of $\bigcup_{1\leq p\leq p}\mathcal{K}_{q}\bigcup\bigcup_{R\in\mathcal{K}_p}\mathcal{C}(R)$ to the algebra generated by the sets  of  the union $\bigcup_{1\leq q\leq p+1}\mathcal{K}_{q}\bigcup\bigcup_{R\in\mathcal{K}_p}\mathcal{C}(R)$ as follows: For $R\in\mathcal{K}_{p}$ and $D\in\mathcal{C}(R)$, we impose that 
\begin{equation}
\eta_{p+1}(R)=\eta_p (R)\text{ and } \eta_{p+1}(D)=\frac{\eta_p (R)}{\# \mathcal{C}(R)} .
\end{equation}

And then, for $R\in\mathcal{K}_p$, $D\in\mathcal{C}(R)$ and  $1\leq k\leq N_D$, we set
\begin{equation}
\label{etar}
  \eta_{p+1}(R_{k}^{(D)})=\eta_{p+1} (D)\cdot\frac{\mu(L_{k}^{(D)})}{\sum_{1\leq k^{\prime}\leq N_D}\mu(L_{k^{\prime}}^{(D)})}.
  \end{equation}
  It is easily checked that properties \textbf{(i)} to \textbf{(viii)} hold for $p+1$ and this ends the induction.
  
  \subsection*{Last step: the Cantor set and some of its properties.}
Set $\mathcal{K}_0 =[0,1]^d$ and $\eta_{0}([0,1]^d)=1.$ Define
$$\mathcal{K}=\bigcup_{p\in\mathbb{N}}\mathcal{K}_p\text{ and }K=\bigcap_{p\in\mathbb{N}}K_p.$$
By construction, item \textbf{(i)} of the recursion implies that $K\subset \limsup_{n\to\infty} R_n$. Now, for each $p\ge 1$, let  $\widetilde \eta_p$ be the element of $\mathcal P([0,1]^d)$  supported on $K_p$ and such that for every $R\in\mathcal K_p$ the restriction of $\eta_p$ to $R$ has $\frac{\eta_p(R)}{\mathcal {L}^d(R)}$ as density with respect to $\mathcal{L}^d_{|R}$.  It is easily seen, due to the separation property of the elements of $\mathcal K_p$, for all $p\in\mathbb{N}$, that $(\widetilde\eta_p)_{p\in \mathbb N^*}$ converges weakly to a Borel probability measure $\eta$  such that $\eta(R)=\eta_{p}(R)$  for all $p\in\mathbb{N}$ and $R\in\mathcal{K}_p$.

\medskip

Note that by construction the following properties hold:\medskip
\begin{itemize}
 \item[•] \textbf{Uniform separation property:} For all $p\in\mathbb{N}$ and $n\in\mathbb{N}$ such that $R_{n}\in \mathcal{K}_p$, if  $n_1 ,  n_2 \in \mathbb{N}$ are such that  $R_{n_1} \neq R_{n_2}\in \mathcal{K}(R_n)=\left\{ R'\in\mathcal{K}_{p+1}: \, R^{\prime}\subset R_n\right\}$, then one has $3B_{n_1}\cap 3B_{n_2}=\emptyset$. Indeed, in the case where  $R_{n_1}$ and $R_{n_2}$ are elements of the same $D\in\mathcal C(R_n)$, this follows from \textbf{(v)}; otherwise, this follows from the fact that two distinct elements $D$ and $D'$ of $\mathcal C(R_n)$ are distant from each other by at least $8\cdot 2^{-p(D)}$, where, as before,   $p(D)$ is the unique integer such that $D\in \mathcal{D}_{p(D)}$. 
 
 \medskip
\item[•] The estimates  (\ref{majorect}) and (\ref{majomesC}) show (by induction) that for all $p\in\mathbb N^*$ and $n\in\mathbb{N}$ such that $R_{n}\in \mathcal{K}_p$ one has 
$$
\eta( R_n)\leq r_n^{\alpha -2\varepsilon_p},
$$
and for all $D\in\mathcal{C}(R_n)$, 
\begin{equation}
\label{C}
\eta(D)\leq \kappa_d \,  r_{n}^{\alpha-2\varepsilon_p +\sum_{1\leq i\leq d}\tau_d-\tau_i}=\kappa_d \Big(r_{n}^{\tau_d}\Big)^{\frac{\alpha -2\varepsilon_p+\sum_{1\leq i\leq d}\tau_d-\tau_i}{\tau_d}}.
\end{equation}
\end{itemize}

\subsection*{Upper bound for the $\eta$-measure of a ball.}

\bigskip
Let $C$ be a ball (recall that it is an Euclidean cube) of side length $r$ contained in $[0,1]^d$. Several cases are distinguished. 
\medskip
 

$\bullet$ When $C$ intersects $K_p$ for at most finitely many $p\in\mathbb{N}$, it is clear that $\eta(C)=0$, and  we set $p_C=+\infty$.

\medskip

$\bullet$ When $C$ intersects a unique rectangle of $\mathcal{K}_p$, say $R_{n_C (p)}$, for infinitely many $p\in\mathbb{N}$, then $\eta(C)\leq \eta(R_{n_C(p)})\leq  r_{n_C (p)}^{\alpha -2\varepsilon_p}$ for infinitely many $p$, so  $\eta(C)=0$. Again, we set $p_C=+\infty$.

\medskip

$\bullet$ Suppose now that we are not in one of the previous cases. There exists $p_{C}\in\mathbb{N}$ such that if $ p \leq p_{C}$, $C$ intersects a unique rectangle of $\mathcal K_{p}$ and if  $p \geq p_{C}+1$, $C$ intersects at least two rectangles of $K_p$. \medskip
 Denote by $R_{n_C}$ the unique rectangle in $\mathcal K_{p_{C}}$ intersecting $C$. Let $v>0$  be such that $r=r_{n_C}^{v}$. Again, several cases are distinguished.\medskip

\noindent \textbf{(i)} Suppose $r\geq r_{n_C}^{\tau_d}$ (i.e. $v\leq \tau_d $):  
Suppose, moreover, that $r<r_{n_C}$, i.e. $1<v\le \tau_d$. Recall that for $D\in\mathcal{C}(R_{n_C})$ one has (see (\ref{C}))
 $$\eta(D)\leq \kappa_d \, \Big(r_{n_C}\Big)^{\alpha -2\varepsilon_{p_C}+\sum_{1\leq i\leq d}\tau_d-\tau_i}.$$
 Also, there exists  $\widetilde{\kappa}_d >0$, depending on $d$ so that 
%
%
 \begin{align*}
\#\left\{D\in\mathcal{C}(R_{n_C}): \ D\cap C\neq \emptyset\right\} & \leq  \widetilde{\kappa}_d \,\prod_{i:  \tau_i<v}\Big(\frac{r_{n_C}^v}{r_{n_C}^{\tau_d}}\Big) \prod_{i: \tau_i\geq v}\Big(\frac{r_{n_C}^{\tau_i}}{r_{n_C}^{\tau_d}}\Big)\\
 &\leq \widetilde{\kappa}_d \, r_{n_C}^{-d\tau_d+\sum_{i:  \tau_i<v}v+\sum_{i:  \tau_i\geq v}\tau_i}.
 \end{align*}
Provided that $\kappa_d$ was chosen larger than $\widetilde{\kappa}_d$ at first, one gets 
 
This gives the following upper bound for $\eta(C)$:
 \begin{align*}
 \eta(C) &\leq \sum_{D\in\mathcal{C}(R_{n_C}):D\cap C\neq \emptyset}\eta(D)\\
 &\leq (\#\left\{D\in\mathcal{C}(R_{n_C}): \ D\cap C\neq \emptyset\right\})\cdot \kappa_d \, \Big(r_{n_C}\Big)^{\alpha -2\varepsilon_{p_C}+\sum_{1\leq i\leq d}\tau_d-\tau_j}\\ 
 &\ \leq  \kappa_d^2  \, \Big(r_{n_C}\Big)^{-2\varepsilon_{p_C}}\,\Big(r_{n_C}\Big)^{-d\tau_d+\alpha+\sum_{i: \tau_i<v}v+\sum_{i: \tau_i\geq v}\tau_i+\sum_{1\leq i\leq d}\tau_d-\tau_i} &\\
 &\ \leq \kappa_d^2 \,\Big(r_{n_C}\Big)^{-2\varepsilon_{p_C}}\, \Big( r_{n_C}\Big)^{\alpha+\sum_{i: \tau_i<v}v-\tau_i}&\\
 &\ \leq \kappa_d^2 \, \Big(r^{-2\varepsilon_{p_C}}\Big)\,r^{\frac{\alpha+\sum_{i: \tau_i<v}v-\tau_i}{v}}.
 \end{align*}
The mapping $f:v\mapsto \frac{\alpha+\sum_{i: \tau_i <v}v-\tau_i}{v}$ reaches its minimum at one of the $\tau_i$, say $\tau_{i_0}$ with $1\leq i_0\leq d$. This can be rephrased as $s(\mu,\boldsymbol{\tau})=\min_{1\leq i\leq d}(\frac{\alpha+\sum_{1\leq j\leq i}\tau_i-\tau_j}{\tau_i})=f(\tau_{i_0})$. It follows that
 \begin{equation}
 \label{majo}
 \eta(C)\leq \kappa_d^2 \, r^{s(\mu,\boldsymbol{\tau})-2\varepsilon_{p_C}}.
 \end{equation}
On the other hand, if  $r\ge r_{n_C}$, i.e. $v\le 1$, then by \eqref{majorect},  one has 
$$
\eta(C)\leq \eta(R_{n_C})\leq r_{n_C}^{\alpha -2\varepsilon_{p_C}}\le r^{\alpha -2\varepsilon_{p_C}},
$$
and  \eqref{majo} holds as well, since $\alpha=f(\tau_1)\ge s(\mu,\boldsymbol{\tau})$. 
 
\medskip
\noindent
\textbf{(ii)} Suppose now that $r< r_{n_C}^{\tau_d}$ (i.e. $v>\tau_d$):
 
 \medskip
 
 Recall that $r_{n_C}^{\tau_d}$ is the length of the smallest side of the rectangle $R_{n_C}$. Since $C$ has side length less than $r_{n_C}^{\tau_d}$, and the side length of the cubes of $\mathcal{C}(R_{n_C})$ is larger than or equal to $C_d^{-1} r_{n_C}^{\tau_d}$, one deduces that $C$ intersects at most $\widetilde C_d$ of those cubes, where $\widetilde C_d$ depends on $d$ only. For all $D\in\mathcal{C}(R_{n_C})$, such that $C\cap D\neq \emptyset$, denote by $R_{k_1}^{(D)},...,R_{k_{N_{C,D}}}^{(D)}$ the rectangles included in $D$ that  intersect $C$. 
\medskip
 
$\bullet$ Suppose first that $20r\leq 2^{-p(D)} \rho_{p_C+1}$ (where $D\in\mathcal{D}_{p(D)}$): Note that for all $1 \leq i\neq j \leq N_{C,D}$, $3B_{k_i}^{(D)} \cap 3B_{k_j}^{(D)}=\emptyset  $. Also,  $C$ intersects both $B_{k_i}^{(D)}$ and $B_{k_j}^{(D)}$, and by construction, since $L_{k_i}^{(D)} \cap L_{k_j}^{(D)}=\emptyset$ and $|L_{k_j}^{(D)}|=4|B_{k_j}^{(D)}|$, we have $r\ge r_{k_j}^{(D)}$.  By  Lemma \ref{geo} applied to each pair $\{A=C,B=B_{k_j}^{(D)}\}$ and $q=3$, one gets  $\bigcup_{1\leq i\leq N_{C,D}} 3B_{k_i}^{(D)}\subset 5C$. In particular, $\bigcup_{1\leq i\leq N_{C,D}} L_{k_i}^{(D)}\subset 10C$ since $L_{k_i}^{(D)}\subset 5B_{k_i}^{(D)}$ for each $i$.  Consequently, 
$$
\sum_{1\leq i\leq N_{C,D}}\mu(L_{k_i}^{(D)})\leq \mu(10C).
$$
 Further recall that, by item \textbf{(iv)} of the recurrence scheme, for any $1\leq i\leq N_{C,D}$ the ball $L_{k_i}^{(D)}$ is centered on $E^{\varepsilon_{p_C+1}}_{D}$.  Thus there is $x\in E^{\varepsilon_{p_C+1}}_{D}\cap 10C$. Since  one has $10C\subset B(x,20r)$ and $\frac{20r}{2^{-p(D)}}\leq \rho_{p_C+1}$,  by \eqref{majoraf} we get
 \begin{equation}
 \label{mua}
 \mu(10C)  \leq \mu(B(x,20r))\leq C_{\mu}\,\mu(D)\,\Big(\frac{20r}{2^{-p(D)}}\Big)^{\alpha-\varepsilon_{p_C+1}}.
 \end{equation}
It follows from \eqref{coverc}, \eqref{etar} and \eqref{mua} that
 \begin{align*}
 \eta(C\cap D)\leq\sum_{1\leq i\leq N_{C,D}}\eta(R_{k_{i}}^{(D)})
 &\leq \eta(D)\sum_{1\leq i\leq N_{C,D}}\frac{\mu(L_{k_{i}}^{(D)})}{\sum_{1\leq j\leq N_{D}}\mu(L_{j}^{(D)})}\\
 &\leq \eta(D)\sum_{1\leq i\leq N_{C,D}}\frac{\mu(L_{k_{i}}^{(D)})}{(4Q_d)^{-1}\mu(D)}\\
 &\leq 4Q_d\, \frac{\eta(D)}{\mu(D)}\mu(10C)\\ &\leq C_{\mu}\, \eta(D)\, 4Q_d \, \Big(\frac{20r}{2^{-p(D)}}\Big)^{\alpha-\varepsilon_{p_C+1}}.
  \end{align*}
This yields
\begin{align*}
 \eta(C)&\leq \sum_{D\in\mathcal{C}(R_{n_C}):C\cap D\neq \emptyset}\eta(C\cap D)\\
 &\leq \widetilde C_d\, C_{\mu} \,\max_{D\in\mathcal{C}(R_{n_C}):C\cap D\neq \emptyset}\eta(D)\, 4Q_d \,\Big(\frac{20r}{2^{-p(D)}}\Big)^{\alpha-\varepsilon_{p_C+1}}.
 \end{align*}
Moreover by (\ref{majomesC}), for each $D\in\mathcal{C}(R_{n_C})$ such that $C\cap D\neq \emptyset$,
$$\eta(D)\leq \kappa_d \,  2^{2p(D)\varepsilon_{ p_C} }\,  2^{-s(\mu,\boldsymbol{\tau}) p(D)},$$ 
hence
 $$
 \eta(C)\leq \widetilde C_d\,   C_{\mu}\, \kappa_d \, 4Q_d \, 2^{2p(D)\varepsilon_{ p_C} }\, \Big(2^{-p(D)}\Big)^{s(\mu,\boldsymbol{\tau})}\,\Big(\frac{20r}{2^{-p(D)}}\Big)^{\alpha-\varepsilon_{p_C+1}}.
 $$
Since $C_d\, 2^{-p(D)}\ge r_{n_C}^{\tau_d}\geq r$ and the sequence $(\varepsilon_p)_{p\ge 1}$ is decreasing and bounded, it follows that for some constant $\gamma$ depending only on the dimension $d$ and $\mu$, one has 
\begin{align*}
 \eta(C)\leq \gamma \, r^{-3\varepsilon_{p_C}}\frac{r^{\alpha}}{2^{-p(D)(\alpha-s(\mu,\boldsymbol{\tau}))}}
 =\gamma \, r^{-3\varepsilon_{p_C}}\left( \frac{r}{2^{-p(D)}}\right)^{\alpha-s(\mu,\boldsymbol{\tau})}  r^{s(\mu,\boldsymbol{\tau})}.
 \end{align*}
Thus, as  $ C_d 2^{-p(D)}\geq r$ and $s(\mu,\boldsymbol{\tau})\leq \alpha$ (so that  $t>0\mapsto t^{\alpha-s(\mu,\boldsymbol{\tau})}$ is non decreasing), we finally obtain
 $$
 \eta(C)\leq \gamma  \, C_d^{\alpha-s(\mu,\boldsymbol{\tau})}\,   r^{s(\mu,\boldsymbol{\tau})-3\varepsilon_{p_C}}\le \gamma  \, C_d^{\alpha}\,   r^{s(\mu,\boldsymbol{\tau})-3\varepsilon_{p_C}}.
 $$
 \medskip
 
 $\bullet$ Suppose now that $20 \rho_{p_C +1}\, 2^{-p(D)}\leq r <r_{n_C}^{\tau_d}$: Again, by definition of $p(D)$, one has $r_{n_C}^{\tau_d}\le C_d 2^{-p(D)}$. Consequently, $C$ is covered by at most $\lfloor(C_d/20 \rho_{p_C +1})+1\rfloor ^d$  cubes of side length $20 \rho_{p_C+1}\,2^{-p(D)} $. Denoting these cubes  by $D_1,\ldots,D_{k}$, and recalling \eqref{sub}, the previous estimate yields
 \medskip
 \begin{align*}
 \ \ \ \ \ \ \ \  \eta(C)\leq  \sum_{i=1}^k\eta(D_i) &\leq \lfloor(C_d/20 \rho_{p_C +1})+1\rfloor ^d\,  \gamma\, \, C_d^{\alpha}\,\Big( 20 \rho_{p_C+1}\,2^{-p(D)}\Big)^{s(\mu,\boldsymbol{\tau})-3\varepsilon_{p_C}}\\
 &\leq \gamma_1 \rho_{p_C +1}^{-d}\, r^{s(\mu,\boldsymbol{\tau})-3\varepsilon_{p_C}}\leq \gamma_1 \, r^{s(\mu,\boldsymbol{\tau})-4\varepsilon_{p_C}}
 \end{align*}  
 for some constant $\gamma_1$ depending only on $d$ and $\mu$ (we used that $20 \rho_{p_C +1}\, 2^{-p(D)}\leq r$ to get the third inequality,  and \textbf{(iii)} as well as the inequality $\varepsilon_{p_C}\ge \varepsilon_{p_C+1}$ to get the fourth one). 
 
\medskip 

To conclude the proof, note that due to the uniform separation property outlined after the last step of the construction of $(K,\eta)$, 
$$p(r)=\inf\{p_C:\,  C\text{ is ball  of radius r included in $[0,1]^d$}\}$$
 tends to $+\infty$ as $r$  tends to $0$. 
 
 Combining the previous estimates, setting $\widetilde{\gamma}_1=\max\left\{\gamma_1 ,\gamma .C_d ^{\alpha},\kappa_d ^2 \right\}$, we finally get
$$
\eta(C)\le \widetilde{\gamma}_1 r^{s(\mu,\boldsymbol{\tau})-4\varepsilon_{p(r)}}.
$$ 
In particular, for any $p\in\mathbb{N}$, setting $r_p =\frac{1}{2}\sup\left\{r:p(r)\leq p\right\}$, it holds that for any $r\leq r_p$, any ball $C$ of radius $r$,
$$
\eta(C)\le \widetilde{\gamma}_1 r^{s(\mu,\boldsymbol{\tau})-4\varepsilon_{p}}.$$

By Lemma \ref{MD}, since $\eta(K)=1$, it holds that

$$\dim_H(K)\geq s(\mu,\boldsymbol{\tau})-4\varepsilon_{p}.$$

Letting $p\to+\infty$ proves Theorem \ref{lowerbound}.

\end{document}